\documentclass[12pt]{amsart}
\usepackage{amssymb}
\usepackage{amsfonts}
\usepackage{eurosym}
\usepackage{amsmath}
\usepackage{amsfonts,xcolor}
\usepackage{amssymb}
\usepackage[all]{xy}
\usepackage{amssymb}
\usepackage{subfigure}
\usepackage{graphicx}
\usepackage[open,openlevel=1]{bookmark}

\theoremstyle{plain}
\newtheorem{acknowledgement}{Acknowledgement}

\newtheorem{corollary}{Corollary}

\newtheorem{proposition}{Proposition}
\newtheorem{remark}{Remark}

\newtheorem{theorem}{Theorem}
\numberwithin{equation}{section}

\begin{document}
\title[Local extrema of spectrograms and Gaussian Analytic Functions]{Local
maxima of white noise spectrograms and Gaussian Entire Functions}
\author{Luis Daniel~Abreu}
\address{NuHAG, Faculty of Mathematics, University of Vienna,
Oskar-Morgenstern-Platz 1, A-1090, Vienna, Austria\\
 }

\address{(Formerly) Acoustics Research Institute, Austrian Academy of Sciences, Wohllebengasse 12-14, A-1040, Vienna, Austria. \\
}

\email{abreuluisdaniel@gmail.com}
\thanks{L.D.~Abreu was supported by FWF Project 31225-N32}
\thanks{}
\date{\today }

\begin{abstract}
We confirm Flandrin's prediction for the expected average of local maxima of
spectrograms of complex white noise with Gaussian windows (Gaussian
spectrograms or, equivalently, modulus of weighted Gaussian Entire
Functions), a consequence of the conjectured double honeycomb mean model for
the network of zeros and local maxima, where the area of local maxima
centered hexagons is three times larger than the area of zero centered
hexagons. More precisely, we show that Gaussian spectrograms, normalized
such that their expected density of zeros is $1$, have an expected density
of $5/3$ critical points, among those $1/3$ are local maxima, and $4/3$
saddle points, and compute the distributions of ordinate values (heights)
for spectrogram local extrema. This is done by first writing the
spectrograms in terms of Gaussian Entire Functions (GEFs). The extrema are
considered under the translation invariant derivative of the Fock space
(which in this case coincides with the Chern connection from complex
differential geometry). We also observe that the critical points of a GEF
are precisely the zeros of a Gaussian random function in the first higher
Landau level. We discuss natural extensions of these Gaussian random
functions: Gaussian Weyl-Heisenberg functions and Gaussian bi-entire
functions. The paper also reviews recent results on the applications of
white noise spectrograms, connections between several developments, and is
partially intended as a pedestrian introduction to the topic.
\end{abstract}

\maketitle

\section{Introduction}

The present paper is concerned with the \emph{critical points }(divided in 
\emph{local maxima},\emph{\ local minima} and \emph{saddle points}) of
spectrogram transformed white noise. More precisely, we are interested in
statistical averages of critical points and local extrema of white noise
spectrograms with Gaussian windows. As first observed in \cite{BFC}, such
white noise spectrograms can be written in terms of a Gaussian Entire
Function (GEF), an observation which had the virtue of connecting the
traditionally applications-oriented field of time-frequency analysis, to the
traditionally more fundamental topic of Gaussian Analytic Functions.

We will both deal with the Gaussian Entire Function%
\begin{equation*}
F(z)=\sum_{k=0}^{\infty }a_{k}\frac{z^{k}}{\sqrt{k!}}
\end{equation*}%
and its translation invariant derivative%
\begin{equation*}
F_{1}(z)=F_{1}(z,\overline{z})=\left( \partial _{z}-\overline{z}\right)
F(z)=\sum_{k=0}^{\infty }a_{k}\frac{(k-\left\vert z\right\vert ^{2})z^{k-1}}{%
\sqrt{k!}}\text{,}
\end{equation*}%
where $a_{k}$ are i. i. d. Gaussian random variables. The zeros of $F_{1}(z)$
describe the critical points of $H(z)=\left\vert F(z)\right\vert
^{2}e^{-\left\vert z\right\vert ^{2}}$, which can be divided in saddle
points and local maxima/minima, while, due to the maximum modulus principle,
the equation $F^{\prime }(z)=0$ can only yield saddle points. As we will see
below, $\left\vert e^{-\left\vert z\right\vert ^{2}/2}F(z)\right\vert ^{2}$
and $\left\vert e^{-\left\vert z\right\vert ^{2}/2}F_{1}(z)\right\vert ^{2}$
can be identified with white noise spectrograms windowed by the Gaussian and
the first Hermite function, respectively. This leads to a Gaussian functions
companion to the determinantal point processes in higher Landau levels \cite%
{HaiAron,APRT,abgrro17,HendHaimi,SHIRAI,MakotoShirai,Deviations}, also
defined using time-frequency transforms. Observe that, while $F(z)$ is
analytic, $F_{1}(z)$ is just polyanalytic, since $\partial _{\overline{z}%
}^{2}F_{1}(z,\overline{z})=0$ (see \cite{AF} for an introduction to
polyanalytic functions from the time-frequency analysis perspective).

We now make an \emph{intermezzo}\ in this introduction to point out two
curious properties.

\begin{itemize}
\item \emph{Silence} points (zeros)\emph{\ repel} each other strongly ($2$%
-point intensity vanishes at zero \cite[Section 3]{GAFbook}).

\item \emph{Loud} points (critical points) display weaker repulsion (\emph{%
positive }$2$-point intensity at zero \cite{Baber}).
\end{itemize}

Back to our work, we will first compute \emph{the expected number of
critical, saddle and local maxima} coordinates, defined under constraints on
the critical points equation,%
\begin{equation*}
F_{1}(z_{c})=0\text{,}
\end{equation*}%
and then the corresponding average distribution of\emph{\ ordinate values}
(the expected average values of $F(z)$ at the critical, saddle and local
maxima coordinates), where the ordinate value $x_{c}$ associated to the
point $z_{c}$ is 
\begin{equation*}
x_{c}=\left\vert e^{-\left\vert z\right\vert ^{2}/2}F(z_{c})\right\vert
=Spec_{g}\mathcal{W}(z_{c})^{1/2}\in \mathbb{R}^{+}\text{.}
\end{equation*}%
The description of the local extrema and corresponding ordinates of
spectrogram transformed white noise is likely to be useful in signal
analysis applications involving local maxima of spectrograms as in, for
instance, \cite{Wang}, but also in reassignment \cite{FlandrinReassignment}
and syncrosqueezing \cite{Dau} techniques. This is particularly evident in
differential reassignment \cite{ReassignmentMaxima}, where local maxima are
attractors of the corresponding vector field.

Our results also complement the information used in the recent algorithms
aimed at recovering a signal embedded in white noise from its spectrogram
zeros \cite{Silence0,Silence,BFC,BH,Pole,PNAS,BardenetSampta,Escudero}
(these `zero-based' methods have also been recently extended to wavelets,
see \cite{KolSapmTA,FilteringWavelet,BH}, to the sphere \cite{Sphere} and to
the Stockwell transform \cite{Stockwell}). Such algorithms have found
notable applications, for instance, in methods for anonymize motion sensor
data and enhance privacy in the Internet of Things (IoT) \cite{IoT}.

The analysis of local extrema of white noise spectrograms relies on the
eigenvalues of the Hessians of the Gaussian vectors, from which one can
derive the Kac-Rice type formulas as in \cite{DZS}. The resulting expected
number of local maxima is $1/3$ times the expected number of zeros. One can
provide heuristics for the $1/3$ proportion using a random honeycomb model,
suggested by Flandrin to describe zeros and local maxima of spectrograms of
white noise (see Figure 1 on page 6). By extrapolating such heuristics to
the setting of holomorphic sections of a Hermitian holomorphic line bundle
over a complex manifold \cite{DZS}, the model also suggests a hand-waving
explanation for the occurrence of the factor $1/3$ in the first term of the
statistics of supersymmetric vacua (modelled as fixed Morse index critical
points of random holomorphic sections). Moving to a different physical
context, a correspondence between the critical points of Gaussian entire
functions and the zeros of a Gaussian bi-entire function will be described.
The Gaussian bi-entire function lives on the first higher Landau level
eigenspace, a space of paramount importance in condensed matter physics,
used to model electron dynamics on a energy level above the first layer of
charged-like particles. The Landau levels are the classical explanation for
the step-like changes in conductivity under the action of a constant
magnetic field associated with the integer Quantum-Hall effect \cite{Nobel}.

The results in this paper are novel to a certain extent, since their
statement (both in the GEF and time-frequency language) has not appeared
before in the literature and given a direct proof. But one should not go as
far as calling them `genuinely new', since general results have been
obtained for compact K\"{a}hler manifolds and the particular calculations
done for $SU(2)$ random polynomials in \cite{DZS} and \cite{FZ1}. It may
come as a surprise that the more simple (and also the most well-studied) case of
the GEF has not been studied before, since the cases treated in \cite{DZS}
and \cite{FZ1} are much more complicated. At least formally, our main
results could have been indirectly derived (one may say `guessed', since
actually some definitions are a bit different for the non-compact case, most
notably in the section on critical values) by properly considering limit
cases of the results for $SU(2)$ polynomials in \cite{DZS} and \cite{FZ1}.
Even if this argument could be made rigorous, it would be a long journey in
comparison to the direct and relatively elementary proofs we offer in this
paper. 

An outline of the paper follows. Since we will alternate between the
spectrogram and GEF formulations, to avoid confusing readers not familiar
with both topics, we try to always present both perspectives, sometimes at
the cost of a little redundancy. In the next section we present a short
review of the required fundamental concepts about random Gaussian functions
and time-frequency analysis. In the third section we present the main
results, together with some remarks concerning the motivations, implications
and connections to the work of other authors. The proofs of the main results
are gathered in section 4. In section 5, we will consider Gaussian \emph{%
non-analytic} functions motivated by the results of the paper. First we
observe that the nonzero critical points of $H(z)$ are precisely the zeros
of a bi-analytic Gaussian function with correlation kernel given by the
reproducing kernel of the second Landau level. This suggests a more general
class of Gaussian random functions with correlation kernels given by the
kernel of the Weyl-Heisenberg ensemble \cite{abgrro17,APRT}.\ We will also
consider Gaussian bi-analytic functions, with correlation kernel given by
the sum of the reproducing kernels of the first two Landau levels. We close
the presentation with a conclusion section, discussing the connections
between the topics and results involved in the paper, and a few
considerations about their symbiotic nature.

\section{Background}

\subsection{Gaussian entire functions and spectrograms of white noise}

Given a window function $g\in L^{2}({\mathbb{R}})$, the short-time Fourier
transform of $f\in L^{2}({\mathbb{R}})$ is 
\begin{equation}
V_{g}f(x,\xi )=\int_{{\mathbb{R}}}f(t)\overline{g(t-x)}e^{-2\pi i\xi
t}dt,\qquad (x,\xi )\in {\mathbb{R}^{2}}.  \label{eq_stft}
\end{equation}%
when $\lVert g\rVert _{2}=1$, the map $V_{g}$ is an isometry between $L^{2}({%
\mathbb{R}})$ and a closed subspace of $L^{2}({\mathbb{R}^{2}})$: 
\begin{equation*}
\lVert V_{g}f\rVert _{{L^{2}({\mathbb{R}^{2}})}}=\lVert f\rVert _{{L^{2}({%
\mathbb{R}})}},\qquad f\in {L^{2}({\mathbb{R}})}.
\end{equation*}%
If we choose the Gaussian function $h_{0}(t)=2^{\frac{1}{4}}e^{-\pi t^{2}}$, 
$t\in {\mathbb{R}}$, as a window in \eqref{eq_stft}, then a simple
calculation shows that, identifying $z=(x,\xi )\in \mathbb{R}^{2}$ with $%
z:=x+i\xi \in {\mathbb{C}}$, 
\begin{equation}
e^{-i\pi x\xi +\frac{\pi }{2}\left\vert z\right\vert ^{2}}V_{h_{0}}f(x,-\xi
)=2^{1/4}\int_{\mathbb{R}}f(t)e^{2\pi tz-\pi t^{2}-\frac{\pi }{2}z^{2}}dt=%
\mathcal{B}f(z),  \label{Bargmann}
\end{equation}%
where $\mathcal{B}f(z)$ is the \emph{Bargmann transform} of $f$ \ \cite%
{Charly}. The Bargmann transform $\mathcal{B}$ is a unitary isomorphism from 
$L^{2}({\mathbb{R}})$ onto the Bargmann-Fock space $\mathcal{F}(\mathbb{C})$
consisting of all entire functions satisfying 
\begin{equation}
\left\Vert F\right\Vert _{\mathcal{F}(\mathbb{C})}^{2}=\int_{\mathbb{C}%
}\left\vert F(z)\right\vert ^{2}e^{-\pi \left\vert z\right\vert
^{2}}dz<\infty .  \label{Focknorm}
\end{equation}%
We will consider Hermite functions normalized as follows: 
\begin{equation*}
h_{r}(t)=\frac{2^{1/4}}{\sqrt{r!}}\left( \frac{-1}{2\sqrt{\pi }}\right)
^{r}e^{\pi t^{2}}\frac{d^{r}}{dt^{r}}\left( e^{-2\pi t^{2}}\right) ,\qquad
r\geq 0,
\end{equation*}%
Let $g(t):=2^{1/4}e^{-\pi t^{2}}$, $t\in {\mathbb{R}}$, be the
one-dimensional, $L^{2}$-normalized Gaussian. The short-time Fourier
transform of the Hermite functions with respect to $g$ is 
\begin{equation}
V_{g}h_{k}(\frac{\bar{z}}{\sqrt{\pi }})=e^{ix\xi }\frac{z^{k}}{\sqrt{k!}}%
e^{-\left\vert z\right\vert ^{2}/2},\qquad k\geq 0\text{.}  \label{hermiteT}
\end{equation}%
One can formally define complex white noise $\mathcal{W}$ by the series
expansion 
\begin{equation}
\mathcal{W(}t\mathcal{)}=\sum_{k=0}^{\infty }a_{k}h_{k}(t)  \label{wn}
\end{equation}%
where $a_{k}$ are i. i. d. Gaussian random variables (the probability of
such sequences being square summable is zero, thus, with probability one,
the series is not convergent in $L^{2}$, but with a little extra effort the
argument can be made rigorous, by defining the sum (\ref{wn}) in a larger
Banach space, see \cite[Section 3]{BH} for a detailed approach). Then, by
linearity of the Short-time Fourier transform $V_{g}$, we have formally from
(\ref{wn}) and (\ref{hermiteT})%
\begin{equation}
V_{g}\mathcal{W}(\frac{\bar{z}}{\sqrt{\pi }})=\sum_{k=0}^{\infty
}a_{k}V_{g}h_{k}(\frac{\bar{z}}{\sqrt{\pi }})=e^{ix\xi }e^{-\left\vert
z\right\vert ^{2}/2}\sum_{k=0}^{\infty }a_{k}\frac{z^{k}}{\sqrt{k!}}
\label{calc}
\end{equation}

\subsection{Gaussian entire functions and Chern critical points}

Gaussian entire functions are defined, for $z\in \mathbb{C}$, as%
\begin{equation}
F(z)=\sum_{k=0}^{\infty }a_{k}\frac{z^{k}}{\sqrt{k!}}\text{,}  \label{GEF}
\end{equation}%
where $a_{k}$ are i. i. d. Gaussian random variables, and can be related to
the Short-time Fourier transform of white noise (see the formal manipulation
(\ref{calc}) and \cite{BFC,BH} for details) as follows: 
\begin{equation*}
V_{g}\mathcal{W}(\frac{\bar{z}}{\sqrt{\pi }})=e^{ix\xi }e^{-\left\vert
z\right\vert ^{2}/2}\sum_{k=0}^{\infty }a_{k}\frac{z^{k}}{\sqrt{k!}}
\end{equation*}%
Since $F(z)$\ is an entire function, all local minima are zeros (see
Proposition 1 in Section 4 for a proof). Moreover, the maximum principle
rules out the possibility of finding local maxima for $\left\vert
F(z)\right\vert $. However, this is possible to do if we look at the extrema
of white noise spectrograms. More precisely, at the critical, saddle and
local maxima points and corresponding ordinate values, of the following
random function:%
\begin{equation}
Spec_{g}\mathcal{W}(z)\mathcal{=}\left\vert V_{g}\mathcal{W}(\frac{\bar{z}}{%
\sqrt{\pi }})\right\vert ^{2}=\left\vert e^{-\left\vert z\right\vert
^{2}/2}\sum_{k=0}^{\infty }a_{k}\frac{z^{k}}{\sqrt{k!}}\right\vert
^{2}=\left\vert e^{-\left\vert z\right\vert ^{2}/2}F(z)\right\vert ^{2}=H(z)
\label{H}
\end{equation}%
where $V_{g}$\ stands for the Short-Time Fourier Transform with a Gaussian
window, or, equivalently, of%
\begin{equation*}
\log \left( Spec_{g}\mathcal{W}(z)\right) ^{\frac{1}{2}}=\log \left\vert
F(z)e^{-\frac{\left\vert z\right\vert ^{2}}{2}}\right\vert =\log \left\vert
F(z)\right\vert -\frac{\left\vert z\right\vert ^{2}}{2}=U(z)\text{.}
\end{equation*}%
The non-zero critical points of $H(z)$ and $U(z)$ are given by the solutions
of the equation%
\begin{equation}
\nabla _{z}^{\prime }F(z)=0\text{,}  \label{Chern}
\end{equation}%
where $\nabla _{z}^{\prime }$ is the translation-invariant derivative (the
symmetric part of the \emph{Chern connection}, see section 4 for more
details):%
\begin{equation*}
\nabla _{z}^{\prime }=\partial _{z}-\overline{z}
\end{equation*}%
leading us to consider the following Gaussian non-Entire Function, 
\begin{equation}
F_{1}(z)=\nabla _{z}^{\prime }F(z)=\sum_{k=0}^{\infty }a_{k}\frac{\nabla
_{z}^{\prime }z^{k}}{\sqrt{k!}}=\sum_{k=0}^{\infty }a_{k}\frac{(\partial
_{z}-\overline{z})z^{k}}{\sqrt{k!}}=\sum_{k=0}^{\infty }a_{k}\frac{%
(k-\left\vert z\right\vert ^{2})z^{k-1}}{\sqrt{k!}}  \label{GnEF}
\end{equation}%
and the associated correlation kernel%
\begin{equation*}
K(z,w)=\mathbb{E}(\nabla _{z}^{\prime }F(z)\overline{\nabla _{z}^{\prime
}F(w)})=L_{1}^{0}(\left\vert z-w\right\vert ^{2})e^{z\overline{w}}\text{,}
\end{equation*}%
where $L_{j}^{\alpha }$ denotes the Laguerre polynomial 
\begin{equation*}
L_{j}^{\alpha }(x)=\sum\limits_{i=0}^{j}(-1)^{i}\binom{j+\alpha }{j-i}\frac{%
x^{i}}{i!},\qquad x\in {\mathbb{R}},\qquad j\geq 0,j+\alpha \geq 0\text{.}
\end{equation*}%
$K(z,w)$ turns out to be the reproducing kernel of the eigenspace linked to
the second eigenvalue of the Landau Laplacian with a magnetic field (the
second Landau Level - see the last section of the paper for more details).
The solutions of (\ref{Chern}) are the \emph{Chern critical points} for
holomorphic functions in $\mathbb{C}$, following the terminology of \cite%
{DZS}, where general Riemann surfaces have been considered.

In what follows, we will use the concept of average number of a point
process $\mathcal{X}$ on $\mathbb{C}$, $\mathcal{N}^{\mathcal{X}}(z)$, well
known in point processes theory as the $1$-point intensity (the most popular
notation is $\rho _{1}$) because it does not measure correlations with other
points of the process. It simply gives the probability that $z\in \mathcal{X}
$. For a precise definition (for details, we suggest consulting the
introductory material in \cite{GAFbook}), one considers the expectation of
the random distribution 
\begin{equation*}
\mathbf{\delta }=\sum_{w\in \mathcal{X}}\delta _{w}\text{,}
\end{equation*}%
which can be defined weakly on test functions as%
\begin{equation*}
\mathbb{E}\left( \left[ \mathbf{\delta },\varphi \right] \right) =\mathcal{%
\int_{\mathbb{C}}\varphi }(z)\mathcal{N}^{\mathcal{X}}(z)d\nu (z)\text{,}
\end{equation*}%
where $d\nu (z)=\frac{1}{\pi }dz$, where $dz\ $stands for Lebesgue measure.
For instance, given $\Omega \subset \mathbb{C}$, denote by $\mathcal{X}%
\left( \Omega \right) $ the number of points to fall in $\Omega $ on a
realization of $\mathcal{X}$. To compute the expected number of points to
fall in $\Omega \subset \mathbb{C}$, one chooses $\mathcal{\varphi }(z)=%
\mathbf{1}_{\Omega }(z)$, leading to 
\begin{equation}
\mathbb{E}\left( \mathcal{X}\left( \Omega \right) \right) =\int_{\Omega }%
\mathcal{N}^{\mathcal{X}}(z)d\nu (z)\text{.}  \label{Expectation}
\end{equation}%
For the point process defined by the zeros of $F(z)$, it is well known (from
a trivial computation using Edelman-Kostlan formula \cite{EdelmanKostlan},
see \cite{GAFbook}) that $\mathcal{N}^{zer}=1$, which means that one zero is
expected per unit area of the plane. It may be instructive to observe that
Formula (\ref{Expectation}) allows to compute the density of the expected
number of points. For instance, for the density of expected number of zeros,
we choose $\Omega $\ as a disc of center $z$ and radius $R$ and obtain from (%
\ref{Expectation}) that $\mathbb{E}\left( \mathcal{X}\left( D(z,R)\right)
\right) =\int_{D(z,R)}1d\nu (z)=R^{2}$. Then, 
\begin{equation*}
\lim\limits_{R\rightarrow \infty }\sup_{z\in \mathbb{C}}\frac{\#\text{points
expected in }D(z,R)}{\left\vert D(z,R)\right\vert }=\lim\limits_{R%
\rightarrow \infty }\frac{\mathbb{E}\left( \mathcal{X}\left( D(0,R)\right)
\right) }{R^{2}}=1\text{.}
\end{equation*}

\section{Results and comments}

\subsection{Critical, saddle and local maxima points}

In the next result we compute the corresponding averages for the critical,
saddle and local maxima points of $H(z)$. We denote such averages by $%
\mathcal{N}^{crit}$, $\mathcal{N}^{sadd}$ and $\mathcal{N}^{\max }$,
respectively.

\begin{theorem}
All local minima of $H(z)$ are zeros of $F(z)$ and none of them is a zero of 
$\ \nabla _{z}^{\prime }F(z)$. The remaining critical points of $H(z)$\emph{%
\ }are zeros of\emph{\ \ }$\nabla _{z}^{\prime }F(z)$. Their average number
is $\mathcal{N}^{crit}=5/3$, which can be divided in an expected average
number of $\mathcal{N}^{\max }=1/3$ local maxima of $H(z)$ and $\mathcal{N}%
^{sadd}=4/3$ saddle points.
\end{theorem}

From (\ref{Expectation}), the expected number of local maxima in a set $%
\Omega \subset \mathbb{C}$ with Lesbegue measure $\left\vert \Omega
\right\vert <\infty $ will then be given by%
\begin{equation*}
\mathbb{E}\left( \mathcal{X}_{\max }\left( \Omega \right) \right) =\frac{1}{3%
}\left\vert \Omega \right\vert \text{.}
\end{equation*}%
Taking into account the identification between spectrograms of complex white
noise with Gaussian windows and Gaussian Entire Functions provided by (\ref%
{H}), we have%
\begin{equation*}
Spec_{g}\mathcal{W}(z)=H(z)
\end{equation*}%
and%
\begin{equation*}
\log \left( Spec_{g}\mathcal{W}(z)^{\frac{1}{2}}\right) =U(z)\text{,}
\end{equation*}%
where $V_{g}$\ stands for the Short-Time Fourier Transform with a Gaussian
window $g(t):=2^{1/4}e^{-\pi t^{2}}$, $t\in {\mathbb{R}}$. Under this
identification, Theorem 1 can be reformulated as:

\begin{corollary}
All local minima of $Spec_{g}\mathcal{W}(z)$ are zeros of $V_{g}\mathcal{W}$%
. The remaining critical points of $Spec_{g}\mathcal{W}(z)$\emph{\ }are
zeros of\emph{\ \ }$\nabla Spec_{g}\mathcal{W}(z)$. Their expected average
number is $5/3$, which can be divided in an expected average number of $1/3$
local maxima and $4/3$ saddle points of $Spec_{g}\mathcal{W}(z)$.
\end{corollary}

This confirms a heuristic prediction of Patrick Flandrin, based on a
conjectured double mean honeycomb model for the distribution of zeros and
local maxima of Gaussian spectrograms, illustrated by the scheme of Figure 1
below. According to such model, zeros define a mean honeycomb structure made
of hexagons, which themselves define a larger honeycomb structure attached
to local maxima, resulting in a distribution where the area of the local
maxima centered hexagons is three times larger than the area of zero
centered hexagons (besides Figure 1, see also \cite[pg 148; Figure 15.9]%
{Flbook} \cite[Fig. 1]{Fl1}, \cite[Formula (22)]{Fl2}). Such a distribution
would have a density of zeros three times bigger than the density of local
maxima, as the above result demonstrates to be true. We note in passing that
Sodin and Tsirelson \cite{AsymptNorm}, have shown that the random set with a
triangular structure (slightly different from the perturbed hexagonal
lattice)%
\begin{equation*}
\Lambda =\left\{ \sqrt{3\pi }(k+li)+ce^{2\pi im/3}\eta _{k,l}:k,l\in \mathbb{%
Z};m=0,1,2\right\} \text{,}
\end{equation*}%
where $\eta _{k,l}$ are i.i.d Gaussian, achieves asymptotic similarity with
the zeros of the GEF. Gathering the discussion of this paragraph, as a
further step towards Flandrin's conjecture, \emph{we conjecture that }$%
\Lambda $\emph{\ achieves asymptotic similarity with the local maxima of the
GEF.}

Hexagonal lattices are ubiquitous in Euclidean point configuration problems.
In time-frequency analysis this is discussed in Feichtinger's survey \cite%
{Landscapes}. An important conjecture by Strohmer and Beaver \cite{SB}
expects the condition number of (deterministic) Gabor frames with Gaussian
windows to be optimized by a hexagonal lattice. Significant progress towards
a proof has been made by Faulhuber and collaborators \cite{FS}, and a
preprint with a full solution of the problem has recently been posted in 
\cite{Hexagonal}.

\begin{figure}[tbp]
\centering
\includegraphics[scale=0.8]{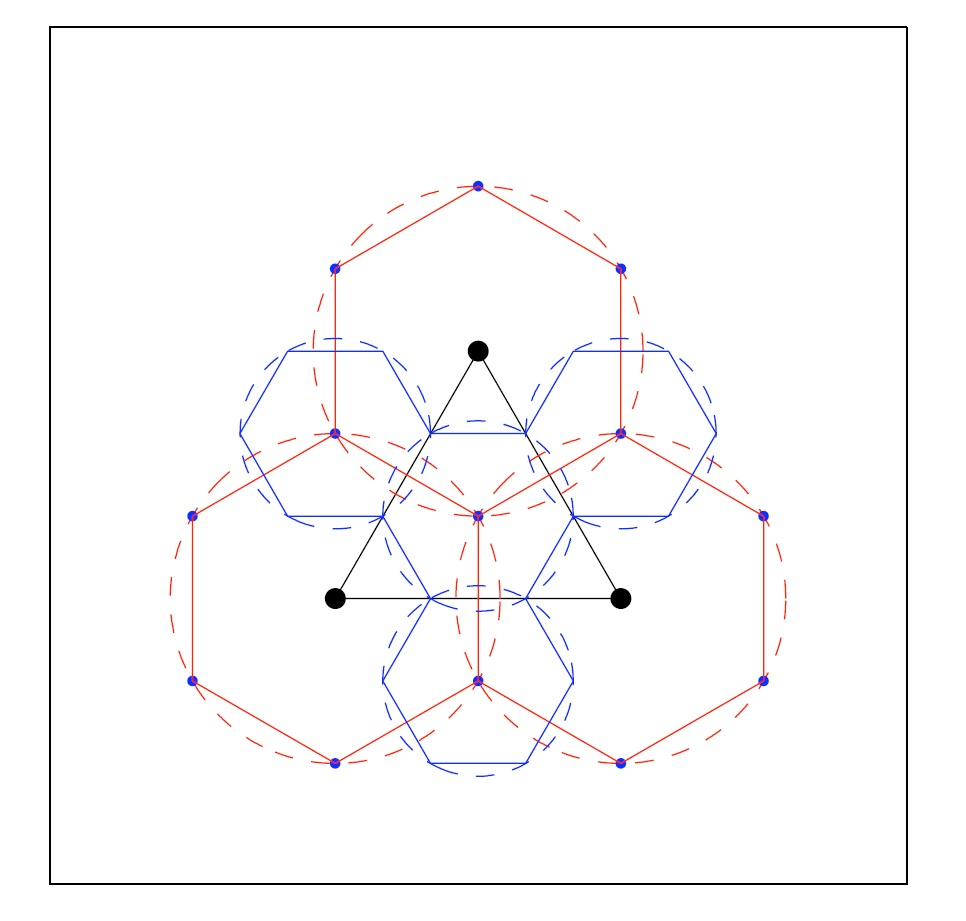}
\caption{Flandrin's honeycomb: the area of the hexagons centered at local
maxima (big dots) is three times the area of the hexagons centered at zeros
(small dots)}
\end{figure}

\begin{remark}
In \cite{BH} the expected number of holomorphic critical points of $F(z)$
(given by the equation $F^{\prime }(z)=0$) has been computed using the
Edelman-Kostlan formula, since $F^{\prime }(z)$ is also an entire function.
It coincides with the large $N$ limit of the number of saddle points of the
chaotic analytic polynomial \cite{Saddles}. While sharing as common aspect
an increase in the number of saddles when compared to the number of zeros,
the result is quite different from the above. The average, explicitly given
as%
\begin{equation*}
N_{hol}^{crit}(z)=\left( 1+(1+\left\vert z\right\vert ^{2})^{-2}\right) 
\text{,}
\end{equation*}%
depends on $\left\vert z\right\vert $, since the zero set of $F^{\prime
}(z)=0$\ is rotation invariant (whence the radial average) but not
translation invariant like $\nabla _{z}^{\prime }F(z)$. In this case all
solutions of the equation $F^{\prime }(z)=0$ yield saddle points, due to the
restriction imposed by the maximum modulus principle of an analytic
function, which prevent the existence of local extrema. By considering $%
\nabla _{z}^{\prime }F(z)=0$ instead (which is equivalent to the zeros of
the gradient of the full spectrogram, as we will see in the next section),
we were able to provide a full description of the local extrema and saddle
points of $H(z)$.
\end{remark}

\begin{remark}
Due to the factor $\left\vert z\right\vert ^{2}$\ in (\ref{GnEF}), $\nabla
_{z}^{\prime }F(z)$ is not an entire function. Thus, one cannot take
advantage of the simplifications leading to the Edelman-Kostlan formula and
the direct use of Kac-Rice type formulas is necessary.
\end{remark}

\begin{remark}
The fact that all local minima of $H(z)$ are zeros of $F(z)$ confirms what
has been observed numerically in \cite{BFC}.
\end{remark}

\begin{remark}
In \cite{Feng}, it has been shown that zeros of $F(z)$ and\ $\nabla
_{z}^{\prime }F(z)$ exhibit repulsion at small scales (when their distance
tends to zero). The impossibility of $F(z)$ and\ $\nabla _{z}^{\prime }F(z)$
having common zeros (see Proposition 1 for a proof) is another manifestation
of such repulsion.
\end{remark}

\begin{remark}
A striking result of Ghosh and Nishry \cite{GN} shows that the zero set of $%
F(z)$, conditioned on the rare event of a disc with no zeros with radius $r$%
, converges (scaled) to an explicit limiting Radon measure with a large
\textquotedblleft forbidden region\textquotedblright\ between a singular
part supported on the boundary of the (scaled) hole and the equilibrium
measure far from the hole. Given the structure of zeros and local maxima
suggested by Flandrin's honeycomb network, partially supported by our
results, it would be interesting to know what happens to the critical/local maxima points
under the same rare event conditioning. At the moment we have no hint on
what to expect. Technically, the main obstruction in adapting the methods
used in \cite{GN} seems to be the absence of Jensen's formula for $\nabla
_{z}^{\prime }F(z)$ (a bi-analytic function). The same obstruction shows up
if one attempts to adapt the known proofs for upper bounding the probability
of existence of a disc without local maxima (the so-called hole probability).
\end{remark}

Also related to the above results are two questions of Nazarov, Sodin and
Volberg in \cite[12.3]{BO2} and \cite[2.2.3]{NS}, where it is also suggested
that the techniques used in \cite{DZS} should also work here. Indeed, the
proof of Theorem 1 consists of a simplification (possible in this case) of
the methods used in \cite{DZS}, where the full calculations have been
performed in one dimension for an elliptic model of Gaussian polynomials. In
our case, results from analysis in one complex variable can be applied to
rule out the possibility of local minima, simplifying the discussion on
Morse indexes (see \cite{DZS} and \cite{FZ1}) to the sign of the Hessian
determinant at critical points. We give now a glimpse of the questions
stated in \cite[12.3]{BO2} and \cite[2.2.3]{NS} (see also the comments
starting on page 146 of \cite{Flbook} for a time-frequency perspective
combined with other ideas, such as differential reassignment and Voronoi
tesselation).

The origin of the problem lies on the idea of gravitational allocation.
Letting the Lebesgue measure fall by gravity and descend along gradient
lines to the zeros of $F(z)$, one is lead to an evolution described by a
stochastic gradient flow involving the log-spectrogram $U(z)$\ as a
potential: 
\begin{equation*}
\frac{dZ(t)}{dt}=-\nabla U(Z(t))\text{.}
\end{equation*}%
Denoting by $\Gamma _{z}$ the corresponding gradient curves that pass
through $z$ (assumed to be neither a zero of $F(z)$ nor a critical point of $%
U(z)$), a random partition of the plane is obtained:%
\begin{equation*}
\mathbb{C}=\bigcup_{a:F(a)=0}B(a)\text{,}
\end{equation*}%
where $B(a)$ is the basin of attraction of $a$: the set of points $z$ such
that $\Gamma _{z}$ terminates at $a$. Now, associate to each $z\in \mathbb{C}
$,\ the basin such that $z\in B(a_{z})$ and define two basins as neighbors
if they have a common gradient curve on their boundary. Then $N_{z}$ is the
number of basins neighbors of $B(a_{z})$, which equals the number of saddle
points of $U(z)$ connected by gradient curves to $a_{z}$. Since almost
surely each saddle is connected with two zeros, \emph{at least heuristically,%
} we have:%
\begin{equation*}
\mathbb{E}N_{z}=_{(1)}2\frac{\text{expected average number of saddle points}%
}{\text{expected number of zeros}}=R^{sadd/zeros}\text{.}
\end{equation*}%
By Theorem 1, $R^{sadd/zeros}=\frac{8}{3}$. Likewise, the number of basins
that meet at the same local maximum is%
\begin{equation*}
\mathbb{E}N_{z}^{\max }=_{(2)}2\frac{\text{expected average number of saddle
points}}{\text{expected number of local maxima}}=R^{sadd/local\max }\text{.}
\end{equation*}%
By Theorem 1, $R^{sadd/local\max }=8$.

\begin{remark}
To provide an answer for the questions in \cite{BO2} it is necessary to show
the identities $\mathbb{(}1\mathbb{)}$ and $\mathbb{(}2\mathbb{)}$.
\end{remark}

In another direction, the connections explored in this note may shed some
light regarding the ubiquitous presence of the factor $1/3$ in the first
terms of the expansion of the number of critical points under the so-called 
\emph{Chern connection }which, for entire functions in the Fock space, boils
down to the operator $\nabla _{z}^{\prime }$. Indeed, in the comments after
formulas (15) and (16) of \cite{DZS}, Douglas, Shiffman and Zelditch mention
that `\emph{It would be interesting to find a heuristic reason for the
factor 1/3}'. Flandrin's mean honeycomb model provides such an heuristic
reason in the planar case, and suggests that a geometric analogue of such
model may exist for the expected distribution of zeros and maxima of random
functions in general complex Riemann surfaces.

\subsection{Critical, saddle and local maxima ordinate values}

The observations leading to the proof of Theorem 1 allow, with a slight
extra effort, to understand the average distribution of the \emph{ordinate
values }of the critical, saddle and local maxima. By an ordinate critical
value we mean the value of $\left( Spec_{g}\mathcal{W}(z)\right) ^{1/2}$, at
the critical point $z_{c}$ 
\begin{equation*}
x_{c}=\left( Spec_{g}\mathcal{W}(z_{c})\right) ^{1/2}=\left\vert
F(z_{c})\right\vert e^{-\frac{\left\vert z_{c}\right\vert ^{2}}{2}}\in 
\mathbb{R}^{+}
\end{equation*}%
Formally, the expected density of the average ordinates of critical values
of local maxima of $\left( Spec_{g}\mathcal{W}(z)\right) ^{1/2}$ , $D(x)$,
is defined weakly on a compact set $\Omega \subset \mathbb{C}$ for a test
function $\psi $ by the identity: 
\begin{equation*}
\mathbb{E}\left\langle \frac{1}{\left\vert \Omega \right\vert }\sum_{\xi \in
\Omega :\nabla _{\mathcal{\xi }}^{\prime }F(\xi )=0}\delta _{\left\vert F(%
\mathcal{\xi })e^{-\frac{\left\vert \mathcal{\xi }\right\vert ^{2}}{2}%
}\right\vert },\psi \right\rangle =\mathbb{E}\left( \frac{1}{\left\vert
\Omega \right\vert }\sum_{\xi \in \Omega :\nabla _{\mathcal{\xi }}^{\prime
}F(\xi )=0}\psi \left\vert F(\mathcal{\xi })e^{-\frac{\left\vert \mathcal{%
\xi }\right\vert ^{2}}{2}}\right\vert \right) =\int_{\mathbb{R}^{+}}\psi
(x)D(x)dx
\end{equation*}%
and similarly for saddles and local maxima. The following theorem, which is
our second main result, describes the expected density distribution of
ordinate values of critical, local maxima and saddle points (which are not
probability distributions). The proof is based on methods of Feng and
Zelditch \cite{FZ,FZ1}.

\begin{theorem}
The expected density of the average ordinate critical values of local maxima
of $\left( Spec_{g}\mathcal{W}(z)\right) ^{1/2}$ is given by the function $%
D^{\max }$, defined on $\left[ 0,\infty \right] $ by 
\begin{equation*}
D^{\max }(x)=2x\left( x^{2}-2+2e^{-\frac{x^{2}}{2}}\right) e^{-x^{2}}\text{,}
\end{equation*}%
of the ordinate values of saddle points, by%
\begin{equation*}
D^{sadd}(x)=4xe^{-\frac{3}{2}x^{2}}\text{,}
\end{equation*}%
and of the ordinates of critical values by 
\begin{equation*}
D^{crit}(x)=D^{sadd}(x)+D^{\max }(x)=2x\left( x^{2}-2+4e^{-\frac{x^{2}}{2}%
}\right) e^{-x^{2}}\text{.}
\end{equation*}
\end{theorem}

\section{Proof of Theorems 1 and 2}

\subsection{Critical points under the Chern connection}

The so called \emph{Chern connection} under the Bargmann-Fock metric is
given as%
\begin{equation*}
\nabla _{z}=\nabla _{z}^{\prime }+\nabla _{z}^{\prime \prime }
\end{equation*}%
with%
\begin{equation*}
\nabla _{z}^{\prime }=\partial _{z}-\overline{z}
\end{equation*}%
and%
\begin{equation*}
\nabla _{z}^{\prime \prime }=\partial _{\overline{z}}\text{.}
\end{equation*}%
Following \cite{DZS},\ we define a critical point of $F(z)$ under the Chern
connection as a solution of 
\begin{equation*}
\nabla _{z}F(z)=0\text{.}
\end{equation*}%
If $F(z)$ is entire as in the case of the GEF, then $\nabla _{z}^{\prime
\prime }F(z)=0$ and $\nabla _{z}F(z)=\nabla _{z}^{\prime }F(z)$, simplifying
the critical points equation:%
\begin{equation*}
\nabla _{z}^{\prime }F(z)=\partial _{z}F(z)-\overline{z}F(z)=0\text{.}
\end{equation*}

\begin{proposition}
The nonzero critical points of $H(z)$ and $U(z)$ are given by the equation $%
\nabla _{z}^{\prime }F(z)=0$. All nonzero critical points which are local
extrema points of $H(z)$ and $U(z)$ are necessarily local maxima. Thus, all
local minima are zeros of $F(z)$ and cannot be critical points given by the
equation $\nabla _{z}^{\prime }F(z)=0$.
\end{proposition}

\begin{proof}
The critical points equation%
\begin{equation*}
0=\left\vert \nabla \log \left\vert F(z)e^{-\frac{\left\vert z\right\vert
^{2}}{2}}\right\vert \right\vert =\left\vert \frac{\partial _{z}F(z)}{F(z)}-%
\overline{z}\right\vert \text{,}
\end{equation*}%
where%
\begin{equation*}
\nabla f=\frac{\partial f}{\partial _{z}}+\overline{\frac{\partial f}{%
\partial _{\overline{z}}}}\text{,}
\end{equation*}%
is equivalent, for $F(z)\neq 0$, to 
\begin{equation*}
\partial _{z}F(z)-\overline{z}F(z)=0\Leftrightarrow \nabla _{z}^{\prime
}F(z)=0
\end{equation*}%
Likewise, the critical point equation%
\begin{equation*}
0=\left\vert \nabla \left\vert F(z)\right\vert ^{2}e^{-\left\vert
z\right\vert ^{2}}\right\vert =2e^{-\left\vert z\right\vert ^{2}}\overline{F}%
(z)(\nabla _{z}F(z))
\end{equation*}%
is equivalent, for $F(z)\neq 0$, to $\nabla _{z}^{\prime }F(z)=0$. Thus, the
critical points of%
\begin{equation*}
H(z)=\left\vert F(z)\right\vert ^{2}e^{-\left\vert z\right\vert ^{2}}
\end{equation*}%
are the same as those of%
\begin{equation*}
U(z)=\log \left\vert F(z)e^{-\frac{\left\vert z\right\vert ^{2}}{2}%
}\right\vert =\log \left\vert F(z)\right\vert -\frac{\left\vert z\right\vert
^{2}}{2}
\end{equation*}%
Now observe that, if $F(z)\neq 0$, the Laplacian of $\log \left\vert F(z)e^{-%
\frac{\left\vert z\right\vert ^{2}}{2}}\right\vert $ is negative: 
\begin{equation*}
\Delta \log \left\vert F(z)e^{-\frac{\left\vert z\right\vert ^{2}}{2}%
}\right\vert =\Delta \log \left\vert F(z)\right\vert -2=-2\text{.}
\end{equation*}%
This implies that $\log \left\vert F(z)e^{-\frac{\left\vert z\right\vert ^{2}%
}{2}}\right\vert $ is super harmonic outside the zero set of $F$. Thus, all
local minima of $\log \left\vert F(z)e^{-\frac{\left\vert z\right\vert ^{2}}{%
2}}\right\vert $ are attained when $F(z)=0$, and the same happens with the
local minima of $\left\vert F(z)\right\vert ^{2}e^{-\left\vert z\right\vert
^{2}}$. As a result, all nonzero critical points which are local extrema
points of $H(z)$ and $U(z)$ are local maxima. Finally, by \cite[Lemma 2.4.1]%
{GAFbook} all the zeros of $F(z)$ are simple. It follows that if $F(z)=0$
then $\nabla _{z}^{\prime }F(z)\neq 0$.
\end{proof}

The proofs of theorems 1 and 2 in the next sections will use Kac-Rice
formulas \cite{level} and Hessians \cite{Adler}. They are inspired by the
scheme of \cite{DZS,FZ,FZ1}, but in our much simple setting, we managed to
obtain a direct, more elementary proof.

\subsection{Proof of Theorem 1}

The average of critical points is given by the expected number of zeros of $%
\nabla _{z}^{\prime }F(z)=0$.\ Since $\nabla _{z}^{\prime }F(z)$ is not
analytic, to compute the expected number of its zeros one cannot use the
Edelman-Kostlan formula, as usually done for Gaussian analytic functions.
The key of the proof is to start with an appropriate Hessian matrix whose
determinant sign allows to separate maxima, minima and saddles, and show
that, at the critical points $\nabla _{z}^{\prime }F(z_{c})=0$, this Hessian
determinant is proportional to $\det DF(z_{c})$, which, under the critical
point $\nabla _{z}^{\prime }F(z_{c})=0$ constraint, equals the determinant
of the complex differential matrix which appears on the Kac-Rice formula for
the average of critical points.

The nonzero critical points of $U(z)=\log \left\vert F(z)e^{-\frac{%
\left\vert z\right\vert ^{2}}{2}}\right\vert $ are given by $\nabla
_{z}^{\prime }F(z)=0$. The same is true for $2U(z)=\log \left\vert F(z)e^{-%
\frac{\left\vert z\right\vert ^{2}}{2}}\right\vert ^{2}$. The corresponding
Hessian matrix is%
\begin{eqnarray*}
D^{Hess}(2U)(z) &=&\left[ 
\begin{array}{cc}
\partial _{z}\overline{\partial _{z}\log \left\vert F(z)e^{-\frac{\left\vert
z\right\vert ^{2}}{2}}\right\vert ^{2}} & \partial _{\overline{z}}\overline{%
\partial _{z}\log \left\vert F(z)e^{-\frac{\left\vert z\right\vert ^{2}}{2}%
}\right\vert ^{2}} \\ 
\partial _{z}\overline{\partial _{\overline{z}}\log \left\vert F(z)e^{-\frac{%
\left\vert z\right\vert ^{2}}{2}}\right\vert ^{2}} & \partial _{\overline{z}}%
\overline{\partial _{\overline{z}}\log \left\vert F(z)e^{-\frac{\left\vert
z\right\vert ^{2}}{2}}\right\vert ^{2}}%
\end{array}%
\right]  \\
&=&\left[ 
\begin{array}{cc}
\partial _{z}\left( \frac{1}{\overline{F(z})}\overline{\nabla _{z}^{\prime
}F(z)}\right)  & \partial _{\overline{z}}\left( \frac{1}{\overline{F(z})}%
\overline{\nabla _{z}^{\prime }F(z)}\right)  \\ 
\partial _{z}\left( \frac{1}{F(z)}\nabla _{z}^{\prime }F(z)\right)  & 
\partial _{\overline{z}}\left( \frac{1}{F(z)}\nabla _{z}^{\prime
}F(z)\right) 
\end{array}%
\right] \text{.}
\end{eqnarray*}%
Observing that%
\begin{equation*}
\partial _{z}\left( \frac{1}{F(z)}\nabla _{z}^{\prime }F(z)\right) =-\frac{%
\partial _{z}F(z)}{F(z)^{2}}\nabla _{z}^{\prime }F(z)+\frac{1}{F(z)}\partial
_{z}\left( \nabla _{z}^{\prime }F(z)\right) \text{,}
\end{equation*}%
one can evaluate the entries of the above matrix at the critical points $%
\nabla _{z}^{\prime }F(z_{c})=0$, leading to%
\begin{eqnarray*}
\det D^{Hess}(2U)(z_{c}) &=&\det \left[ 
\begin{array}{cc}
\frac{1}{\overline{F(z_{c}})}\partial _{z}\overline{\nabla _{z}^{\prime
}F(z_{c})} & \frac{1}{\overline{F(z_{c}})}\partial _{\overline{z}}\overline{%
\nabla _{z}^{\prime }F(z_{c})} \\ 
\frac{1}{F(z_{c})}\partial _{z}\nabla _{z}^{\prime }F(z_{c}) & \frac{1}{%
F(z_{c})}\partial _{\overline{z}}\nabla _{z}^{\prime }F(z_{c})%
\end{array}%
\right]  \\
&=&\frac{1}{\left\vert F(z_{c})\right\vert ^{2}}\det \left[ 
\begin{array}{cc}
\nabla _{\overline{z}}^{\prime \prime }\overline{\nabla _{z}^{\prime
}F(z_{c})} & \nabla _{\overline{z}}^{\prime }\overline{\nabla _{z}^{\prime
}F(z_{c})} \\ 
\nabla _{z}^{\prime }\nabla _{z}^{\prime }F(z_{c}) & \nabla _{z}^{\prime
\prime }\nabla _{z}^{\prime }F(z_{c})%
\end{array}%
\right] \text{.}
\end{eqnarray*}%
We have concluded that 
\begin{equation}
\left\vert F(z_{c})\right\vert ^{2}\det D^{Hess}(2U)(z_{c})=\left\vert
\nabla _{z}^{\prime \prime }\nabla _{z}^{\prime }F(z_{c})\right\vert
^{2}-\left\vert \nabla _{z}^{\prime }\nabla _{z}^{\prime
}F(z_{c})\right\vert ^{2}  \label{crit}
\end{equation}%
coincides with the determinant of $D(F)(z_{c})$ at the critical points $%
\nabla _{z}^{\prime }F(z_{c})=0$, where%
\begin{equation}
D(F)(z)=\left[ 
\begin{array}{cc}
\nabla _{\overline{z}}^{\prime \prime }\overline{\nabla _{z}^{\prime }F(z)}
& \nabla _{\overline{z}}^{\prime }\overline{\nabla _{z}^{\prime }F(z)} \\ 
\nabla _{z}^{\prime }\nabla _{z}^{\prime }F(z) & \nabla _{z}^{\prime \prime
}\nabla _{z}^{\prime }F(z)%
\end{array}%
\right] \text{.}  \label{Dhess}
\end{equation}%
Thus, $\det D^{Hess}(2U)(z_{c})$ and $\det DF(z_{c})$ have the same sign.
This allows to separate maxima, minima and saddles as follows. If $\det
DF(z)>0$,\ then both eigenvalues of $DF(z)$ (and consequently, of $\det
D^{Hess}(2U)(z_{c})$) are negative (if they were both positive then $z$
would be a local minima, and this is not possible by Proposition 2.1). Thus,
one has a non-zero local extrema that cannot be a local minimum and $z$ is a 
\emph{local maximum}. Thus, $z$ is a local maximum if 
\begin{equation*}
\det DF(z)>0\text{,}
\end{equation*}%
or%
\begin{equation}
\left\vert \nabla _{z}^{\prime \prime }\nabla _{z}^{\prime }F(z)\right\vert
^{2}>\left\vert \nabla _{z}^{\prime }\nabla _{z}^{\prime }F(z)\right\vert
^{2}\text{.}  \label{localmaxima}
\end{equation}%
Likewise, if 
\begin{equation*}
\det DF(z)<0\text{,}
\end{equation*}%
or, equivalently, if%
\begin{equation}
\left\vert \nabla _{z}^{\prime \prime }\nabla _{z}^{\prime }F(z)\right\vert
^{2}<\left\vert \nabla _{z}^{\prime }\nabla _{z}^{\prime }F(z)\right\vert
^{2}\text{,}  \label{saddle}
\end{equation}%
then $DF(z)$ has both positive and negative eigenvalues and $z$ is a \emph{%
saddle point}. Thus, saddle points are characterized by the condition (\ref%
{saddle}). Furthermore, at a critical point $\nabla _{z}^{\prime }F(z_{c})=0$%
,%
\begin{equation*}
\det DF(z_{c})=\det \left[ 
\begin{array}{cc}
\nabla _{\overline{z}}^{\prime \prime }\overline{\nabla _{z}^{\prime
}F(z_{c})} & \nabla _{\overline{z}}^{\prime }\overline{\nabla _{z}^{\prime
}F(z_{c})} \\ 
\nabla _{z}^{\prime }\nabla _{z}^{\prime }F(z_{c}) & \nabla _{z}^{\prime
\prime }\nabla _{z}^{\prime }F(z_{c})%
\end{array}%
\right] =\det \left[ 
\begin{array}{cc}
\partial _{z}\overline{\nabla _{z}^{\prime }F(z_{c})} & \partial _{z}%
\overline{\nabla _{z}^{\prime }F(z_{c})} \\ 
\partial _{\overline{z}}\nabla _{z}^{\prime }F(z_{c}) & \partial _{\overline{%
z}}\overline{\nabla _{z}^{\prime }F(z_{c})}%
\end{array}%
\right] \text{,}
\end{equation*}%
We will now consider the 3-dimensional Gaussian process 
\begin{equation}
(\nabla _{z}^{\prime }F(z),\nabla _{z}^{\prime }\nabla _{z}^{\prime
}F(z),\nabla _{z}^{\prime \prime }\nabla _{z}^{\prime }F(z))  \label{vector}
\end{equation}%
with correlation kernel 
\begin{equation*}
K(z,w)=\mathbb{E}(\nabla _{z}^{\prime }F(z)\nabla _{z}^{\prime }F(w))=\nabla
_{z}^{\prime }\nabla _{\overline{w}}^{\prime }e^{z\overline{w}%
}=L_{1}^{0}(\left\vert z-w\right\vert ^{2})e^{z\overline{w}}\text{,}
\end{equation*}%
Using the vector (\ref{vector}), the Kac-Rice formula for the average of
critical points can be given as the conditional expectation of $\det 
\mathcal{D}(\nabla _{z}^{\prime }F)$ at a critical point as in \cite[(90)]%
{DZS} or using \cite[Theorem 6.2]{level} as in \cite[Lemma 2.3]{GWHF} (the $%
\frac{1}{\pi }$ factor is included in the measure $d\nu (z)$ of (\ref%
{Expectation})). Thus,%
\begin{equation}
\mathcal{N}^{crit}=\mathbb{E}(\left\vert \det DF\right\vert \delta
_{0}(\nabla _{z}^{\prime }F))\text{.}  \label{Ecritic}
\end{equation}%
Gathering the contents of this discussion, once we compute the Gaussian
density $p(0,v,u)$, by (\ref{localmaxima}), the expected number of local
maxima can be obtained using the following Kac-Rice formula:%
\begin{eqnarray}
\mathcal{N}^{\max } &=&\mathbb{E}(\left\vert \det D(F):\det
D(F)>0\right\vert \delta _{0}(\nabla _{z}^{\prime }F))  \notag \\
&=&\mathbb{E}\left[ \left\vert \left\vert \nabla _{z}^{\prime \prime }\nabla
_{z}^{\prime }F\right\vert ^{2}-\left\vert \nabla _{z}^{\prime }\nabla
_{z}^{\prime }F\right\vert ^{2}:\left\vert \nabla _{z}^{\prime \prime
}\nabla _{z}^{\prime }F(z)\right\vert ^{2}>\left\vert \nabla _{z}^{\prime
}\nabla _{z}^{\prime }F(z)\right\vert ^{2}\right\vert \delta _{0}(\nabla
_{z}^{\prime }F)\right]   \notag \\
&=&\int_{\mathbb{C}^{2}\cap \{\left\vert u\right\vert ^{2}>\left\vert
v\right\vert ^{2}\}}(\left\vert u\right\vert ^{2}-\left\vert v\right\vert
^{2})p(0,v,u)dvdu\text{.}  \label{max}
\end{eqnarray}%
and from (\ref{saddle}), the expected number of saddle points is%
\begin{eqnarray}
\mathcal{N}^{sadd} &=&\mathbb{E}(\left\vert \det D(F):\det D(F)<0\right\vert
\delta _{0}(\nabla _{z}^{\prime }F))  \notag \\
&=&\mathbb{E}\left[ \left\vert \left\vert \nabla _{z}^{\prime \prime }\nabla
_{z}^{\prime }F\right\vert ^{2}-\left\vert \nabla _{z}^{\prime }\nabla
_{z}^{\prime }F\right\vert ^{2}:\left\vert \nabla _{z}^{\prime \prime
}\nabla _{z}^{\prime }F(z)\right\vert ^{2}<\left\vert \nabla _{z}^{\prime
}\nabla _{z}^{\prime }F(z)\right\vert ^{2}\right\vert \delta _{0}(\nabla
_{z}^{\prime }F)\right]   \notag \\
&=&\int_{\mathbb{C}^{2}\cap \{\left\vert u\right\vert ^{2}<\left\vert
v\right\vert ^{2}\}}(\left\vert v\right\vert ^{2}-\left\vert u\right\vert
^{2})p(0,v,u)dvdu\text{.}  \label{saddl}
\end{eqnarray}%
The Gaussian measure $p(0,v,u)$ is given in terms of the covariance matrix 
\begin{equation*}
\Lambda =%
\begin{pmatrix}
1 \\ 
\nabla _{z}^{\prime } \\ 
\nabla _{z}^{\prime \prime }%
\end{pmatrix}%
\begin{pmatrix}
1 & \nabla _{\overline{w}}^{\prime } & \nabla _{\overline{w}}^{\prime \prime
}%
\end{pmatrix}%
K(z,w)=%
\begin{pmatrix}
1 & \nabla _{\overline{w}}^{\prime } & \nabla _{\overline{w}}^{\prime \prime
} \\ 
\nabla _{z}^{\prime } & \nabla _{z}^{\prime }\nabla _{\overline{w}}^{\prime }
& \nabla _{z}^{\prime }\nabla _{\overline{w}}^{\prime \prime } \\ 
\nabla _{z}^{\prime \prime } & \nabla _{z}^{\prime \prime }\nabla _{%
\overline{w}}^{\prime } & \nabla _{z}^{\prime \prime }\nabla _{\overline{w}%
}^{\prime \prime }%
\end{pmatrix}%
K(z,w)
\end{equation*}%
and its inverse by 
\begin{equation*}
p(0,v,u)=\frac{e^{-\left\langle \left( 
\begin{array}{c}
0 \\ 
v \\ 
u%
\end{array}%
\right) ,\Lambda ^{-1}\left( 
\begin{array}{c}
0 \\ 
\overline{v} \\ 
\overline{u}%
\end{array}%
\right) \right\rangle }}{\det \Lambda }
\end{equation*}%
To compute $p(0,v,u)$ we observe that 
\begin{eqnarray*}
K(z,z) &=&e^{\left\vert z\right\vert ^{2}} \\
\nabla _{z}^{\prime }K(z,w)_{|z=w} &=&0 \\
\nabla _{z}^{\prime \prime }K(z,w)_{|z=w} &=&0 \\
\nabla _{z}^{\prime \prime }\overline{\nabla _{w}^{\prime \prime }}%
K(z,w)_{|z=w} &=&e^{\left\vert z\right\vert ^{2}} \\
\nabla _{z}^{\prime }\overline{\nabla _{w}^{\prime }}K(z,w)_{|z=w}
&=&2e^{\left\vert z\right\vert ^{2}} \\
\nabla _{z}^{\prime }\overline{\nabla _{w}^{\prime \prime }}K(z,w)_{|z=w}
&=&0\text{,}
\end{eqnarray*}%
and we multiply the matrix $\Lambda $ by $e^{-\left\vert z\right\vert ^{2}}$%
, since this keeps the measure $p(0,v,u)dvdu$ invariant, leading to the
following covariance matrix and respective inverse:

\begin{equation*}
\Lambda =\left[ 
\begin{array}{ccc}
1 & 0 & 0 \\ 
0 & 2 & 0 \\ 
0 & 0 & 1%
\end{array}%
\right] \text{, \ \ \ \ \ \ \ }\Lambda ^{-1}=\left[ 
\begin{array}{ccc}
1 & 0 & 0 \\ 
0 & \frac{1}{2} & 0 \\ 
0 & 0 & 1%
\end{array}%
\right]
\end{equation*}%
We arrive thus at%
\begin{equation*}
p(0,v,u)=\frac{1}{2}e^{-\frac{\left\vert v\right\vert ^{2}}{2}-\left\vert
u\right\vert ^{2}}\text{.}
\end{equation*}%
For the local maxima, we obtain from (\ref{max}):%
\begin{eqnarray*}
\mathcal{N}^{\max }(z) &=&\frac{1}{2}\int_{\mathbb{C}^{2}\cap \{\left\vert
u\right\vert ^{2}>\left\vert v\right\vert ^{2}\}}(\left\vert u\right\vert
^{2}-\left\vert v\right\vert ^{2})e^{-\frac{\left\vert v\right\vert ^{2}}{2}%
-\left\vert u\right\vert ^{2}}dvdu \\
&=&\int_{x>0}\int_{t>2x}(t-2x)e^{-t-x}dtdx \\
&=&\int_{\mathbb{R}^{+}}e^{-3x}dx=\frac{1}{3}\text{.}
\end{eqnarray*}%
Likewise, (\ref{saddl}) gives the average number of saddle points:%
\begin{eqnarray*}
\mathcal{N}^{sadd}(z) &=&\frac{1}{2}\int_{\mathbb{C}^{2}\cap \{\left\vert
u\right\vert ^{2}<\left\vert v\right\vert ^{2}\}}(\left\vert v\right\vert
^{2}-\left\vert u\right\vert ^{2})e^{-\frac{\left\vert v\right\vert ^{2}}{2}%
-\left\vert u\right\vert ^{2}}dvdu \\
&=&\int_{x>0}\int_{t<2x}(2x-t)e^{-t-x}dtdx \\
&=&2\int_{\mathbb{R}^{+}}e^{-3/2x}dx=\frac{4}{3}\text{.}
\end{eqnarray*}%
Now, from Proposition 1 all nonzero critical points which are local extrema
points of $H(z)$ and $U(z)$ are necessarily local maxima. Thus, critical
points can only be saddles or local maxima and the expected number of
critical points is:%
\begin{equation*}
\mathcal{N}^{crit}=\mathcal{N}^{sadd}+\mathcal{N}^{\max }=1/3+4/3=5/3\text{.}
\end{equation*}%
Thus, we have an average of $5/3$ critical points. Among those, $1/3$ are
local maxima and $4/3$ saddle points.

\subsection{Proof of Theorem 2}

We consider the empirical point measure defined by the values of $\left\vert
F(\mathcal{\xi })\right\vert e^{-\frac{\left\vert \mathcal{\xi }\right\vert
^{2}}{2}}$ of critical points for $\xi $ in a compact set $\Omega \subset 
\mathbb{C}$, as the finite sum normalized by the Lesbegue measure of $\Omega 
$, $\left\vert \Omega \right\vert $: 
\begin{equation*}
\frac{1}{\left\vert \Omega \right\vert }\sum_{\xi \in \Omega :\nabla _{%
\mathcal{\xi }}^{\prime }F(\xi )=0}\delta _{\left\vert F(\mathcal{\xi })e^{-%
\frac{\left\vert \mathcal{\xi }\right\vert ^{2}}{2}}\right\vert }\text{.}
\end{equation*}%
The density of the ordinates of critical values, $D(x)$, is defined weakly
by the identity: 
\begin{equation*}
\mathbb{E}\left\langle \frac{1}{\left\vert \Omega \right\vert }\sum_{\xi \in
\Omega :\nabla _{\mathcal{\xi }}^{\prime }F(\xi )=0}\delta _{\left\vert F(%
\mathcal{\xi })e^{-\frac{\left\vert \mathcal{\xi }\right\vert ^{2}}{2}%
}\right\vert },\psi \right\rangle =\mathbb{E}\left( \frac{1}{\left\vert
\Omega \right\vert }\sum_{\xi \in \Omega :\nabla _{\mathcal{\xi }}^{\prime
}F(\xi )=0}\psi \left\vert F(\mathcal{\xi })e^{-\frac{\left\vert \mathcal{%
\xi }\right\vert ^{2}}{2}}\right\vert \right) =\int_{\mathbb{R}^{+}}\psi
(x)D(x)dx
\end{equation*}%
We first consider the density $D{}^{\wedge }(u)$, defined by%
\begin{equation*}
\mathbb{E}\left\langle \frac{1}{\left\vert \Omega \right\vert }\sum_{\xi \in
\Omega :\nabla _{\mathcal{\xi }}^{\prime }F(\xi )=0}\delta _{F(\mathcal{\xi }%
)e^{-\frac{\left\vert \mathcal{\xi }\right\vert ^{2}}{2}}},\psi
\right\rangle =\int_{\mathbb{C}}\psi (u)D^{\wedge }(u)du
\end{equation*}%
and then we will integrate out the angle variable by $D(x)=\int_{0}^{2\pi
}D^{\wedge }(x,\theta )xd\theta $, after observing that, for radial $\psi $, 
\begin{equation}
\int_{\mathbb{R}^{+}}\psi (x)D(x)dx=\mathbb{E}\left( \frac{1}{\left\vert
\Omega \right\vert }\sum_{\xi \in \Omega :\nabla _{\mathcal{\xi }}^{\prime
}F(\xi )=0}\psi \left\vert F(\mathcal{\xi })e^{-\frac{\left\vert \mathcal{%
\xi }\right\vert ^{2}}{2}}\right\vert \right) =\int_{\mathbb{C}}\psi
(u)D^{\wedge }(u)du\text{.}  \label{D}
\end{equation}%
When $\nabla _{z}^{\prime }F(z)=0$, 
\begin{equation*}
\nabla _{z}^{\prime }(e^{-\frac{\left\vert z\right\vert ^{2}}{2}}\nabla
_{z}^{\prime }F(z))=e^{-\frac{\left\vert z\right\vert ^{2}}{2}}\nabla
_{z}^{\prime }\nabla _{z}^{\prime }F(z)-\frac{\overline{z}}{2}e^{-\frac{%
\left\vert z\right\vert ^{2}}{2}}\nabla _{z}^{\prime }F(z)=e^{-\frac{%
\left\vert z\right\vert ^{2}}{2}}\nabla _{z}^{\prime }\nabla _{z}^{\prime
}F(z)
\end{equation*}%
and%
\begin{equation*}
\nabla _{z}^{\prime \prime }e^{-\frac{\left\vert z\right\vert ^{2}}{2}%
}\nabla _{z}^{\prime }F(z)=e^{-\frac{\left\vert z\right\vert ^{2}}{2}}\nabla
_{z}^{\prime \prime }\left( \partial _{z}F(z)-\overline{z}F(z)\right) -\frac{%
z}{2}e^{-\frac{\left\vert z\right\vert ^{2}}{2}}\nabla _{z}^{\prime
}F(z)=e^{-\frac{\left\vert z\right\vert ^{2}}{2}}\nabla _{z}^{\prime \prime
}\nabla _{z}^{\prime }F(z)=-e^{-\frac{\left\vert z\right\vert ^{2}}{2}}F(z)%
\text{,}
\end{equation*}%
so that we can write the 3-dimensional Gaussian process 
\begin{equation*}
(e^{-\frac{\left\vert z\right\vert ^{2}}{2}}\nabla _{z}^{\prime }F(z),\nabla
_{z}^{\prime }(e^{-\frac{\left\vert z\right\vert ^{2}}{2}}\nabla
_{z}^{\prime }F(z)),\nabla _{z}^{\prime \prime }(e^{-\frac{\left\vert
z\right\vert ^{2}}{2}}\nabla _{z}^{\prime }F(z)))=(e^{-\frac{\left\vert
z\right\vert ^{2}}{2}}\nabla _{z}^{\prime }F(z),e^{-\frac{\left\vert
z\right\vert ^{2}}{2}}\nabla _{z}^{\prime }\nabla _{z}^{\prime }F(z),-e^{-%
\frac{\left\vert z\right\vert ^{2}}{2}}F(z))\text{,}
\end{equation*}%
The Kac-Rice formula is not changed by the resulting multiplication of the
correlation kernel by $e^{-\frac{\left\vert z\right\vert ^{2}}{2}-\frac{%
\left\vert w\right\vert ^{2}}{2}}$. To obtain the density of the ordinates
of critical values, using \cite[Theorem 6.4]{level} gives%
\begin{eqnarray*}
&&\mathbb{E}\left( \frac{1}{\left\vert \Omega \right\vert }\sum_{\xi \in
\Omega :\nabla _{\mathcal{\xi }}^{\prime }F(\xi )=0}\psi \left( F(\mathcal{%
\xi })e^{-\frac{\left\vert \mathcal{\xi }\right\vert ^{2}}{2}}\right)
\right)  \\
&=&\frac{1}{\pi \left\vert \Omega \right\vert }\int_{\Omega }\mathbb{E}\left[
\left\vert \left\vert e^{-\frac{\left\vert \mathcal{\xi }\right\vert ^{2}}{2}%
}\nabla _{\mathcal{\xi }}^{\prime }\nabla _{\mathcal{\xi }}^{\prime }F(%
\mathcal{\xi })\right\vert ^{2}-\left\vert e^{-\frac{\left\vert \mathcal{\xi 
}\right\vert ^{2}}{2}}F(\mathcal{\xi })\right\vert ^{2}\right\vert \psi (e^{-%
\frac{\left\vert \mathcal{\xi }\right\vert ^{2}}{2}}F(\mathcal{\xi }))\delta
_{0}(\nabla _{\mathcal{\xi }}^{\prime }F(\mathcal{\xi }))\right] d\mathcal{%
\xi } \\
&=&\frac{1}{\pi \left\vert \Omega \right\vert }\int_{\Omega }\int_{\mathbb{C}%
^{2}}\psi (u)\left\vert \left\vert v\right\vert ^{2}-\left\vert u\right\vert
^{2}\right\vert p(0,v,u)d\mathcal{\xi }dvdu\text{.} \\
&=&\int_{\mathbb{C}}\psi (u)\left[ \frac{1}{\pi \left\vert \Omega
\right\vert }\int_{\Omega }\int_{\mathbb{C}}\left\vert \left\vert
v\right\vert ^{2}-\left\vert u\right\vert ^{2}\right\vert p(0,v,u)d\mathcal{%
\xi }dv\right] du
\end{eqnarray*}%
Thus, from (\ref{D}),%
\begin{equation*}
D^{\wedge }(u)=\frac{1}{\pi \left\vert \Omega \right\vert }\int_{\Omega
}\int_{\mathbb{C}}\left\vert \left\vert v\right\vert ^{2}-\left\vert
u\right\vert ^{2}\right\vert p(0,v,u)d\xi dv
\end{equation*}%
and, from the previous section,%
\begin{equation*}
p(0,v,u)=\frac{1}{2}e^{-\frac{\left\vert v\right\vert ^{2}}{2}-\left\vert
u\right\vert ^{2}}\text{.}
\end{equation*}%
Writing $u$ in polar coordinates $(x,\theta )$ as%
\begin{equation*}
u=e^{i\theta }x\text{, }x\geq 0\text{,}
\end{equation*}%
we obtain the expected density as a function of the ordinates $x\in \mathbb{R%
}^{+}$:%
\begin{equation*}
D(x)=\int_{0}^{2\pi }D^{\wedge }(x,\theta )xd\theta \text{.}
\end{equation*}%
Now,%
\begin{eqnarray*}
D(x) &=&\frac{1}{2\pi \left\vert \Omega \right\vert }\int_{0}^{2\pi
}\int_{\Omega }\int_{\mathbb{C}}\left\vert \left\vert v\right\vert
^{2}-\left\vert x\right\vert ^{2}\right\vert xe^{-\frac{\left\vert
v\right\vert ^{2}}{2}-\left\vert x\right\vert ^{2}}dvd\xi d\theta  \\
&=&2xe^{-\left\vert x\right\vert ^{2}}\int_{\mathbb{C}}\left\vert
2\left\vert v\right\vert ^{2}-\left\vert x\right\vert ^{2}\right\vert
e^{-\left\vert v\right\vert ^{2}}dv \\
&=&2xe^{-\left\vert x\right\vert ^{2}}\int_{t>0}\left\vert 2t-\left\vert
x\right\vert ^{2}\right\vert e^{-t}dt \\
&=&2xe^{-\left\vert x\right\vert ^{2}}\left( \int_{t>\frac{\left\vert
x\right\vert ^{2}}{2}}\left( 2t-\left\vert x\right\vert ^{2}\right)
e^{-t}dt+\int_{t<\frac{\left\vert x\right\vert ^{2}}{2}}\left( \left\vert
x\right\vert ^{2}-2t\right) e^{-t}dt\right) \text{.}
\end{eqnarray*}%
For the first integral in brackets,%
\begin{eqnarray*}
\int_{t>\frac{\left\vert x\right\vert ^{2}}{2}}\left( 2t-\left\vert
x\right\vert ^{2}\right) e^{-t}dt &=&\int_{t>\frac{\left\vert x\right\vert
^{2}}{2}}2te^{-t}dt-\left\vert x\right\vert ^{2}\int_{t>\frac{\left\vert
x\right\vert ^{2}}{2}}e^{-t}dt \\
&=&2(1+\frac{\left\vert x\right\vert ^{2}}{2})e^{-\frac{\left\vert
x\right\vert ^{2}}{2}}-\left\vert x\right\vert ^{2}e^{-\frac{\left\vert
x\right\vert ^{2}}{2}} \\
&=&2e^{-\frac{\left\vert x\right\vert ^{2}}{2}}
\end{eqnarray*}%
while the second gives%
\begin{equation*}
\int_{t<\frac{\left\vert x\right\vert ^{2}}{2}}\left( \left\vert
x\right\vert ^{2}-2t\right) e^{-t}dt=\left\vert x\right\vert ^{2}+2e^{-\frac{%
\left\vert x\right\vert ^{2}}{2}}-2\text{.}
\end{equation*}%
Thus, using the arguments of the previous section to separate saddle points
from local maxima, we have%
\begin{equation*}
D^{sadd}(x)=4xe^{-\frac{3}{2}x^{2}}
\end{equation*}%
and%
\begin{equation*}
D^{\max }(x)=2x\left( x^{2}-2+2e^{-\frac{x^{2}}{2}}\right) e^{-x^{2}}\text{.}
\end{equation*}%
The total density of critical values is then given by%
\begin{equation*}
D^{crit}(x)=D^{sadd}(x)+D^{\max }(x)=2x\left( x^{2}-2+4e^{-\frac{x^{2}}{2}%
}\right) e^{-x^{2}}\text{.}
\end{equation*}

\section{Gaussian non-analytic functions}

\subsection{Gaussian functions in the second Landau level}

In this section we will rephrase the results on zeros and critical points of
a GEF within the realm of the euclidean Landau levels, a popular
mathematical physics model used to describe the quantum dynamics of
charged-like particles, which, under the presence of a constant
perpendicular magnetic field (which will be assumed constant $=1$), organize
themselves in layers with unit increase of energy measured by the
corresponding eigenvalues (the terminology Landau levels is mostly used to
refer to such eigenvalues). Each such Landau level eigenvalue $r$ has an
associated Hilbert space with reproducing kernel, which can be used to
define planar wave functions for repulsive interacting particles distributed
in a particular energy level. This leads to an integer indexed sequence of
determinantal point processes in the plane \cite{SHIRAI}, an important class
of Weyl-Heisenberg ensembles \cite{APRT}, which boil down to the Ginibre
ensemble in the case $r=0$. The increase in energy required for the
allocation of particles in higher Landau levels was for a long time the
mainstream physical explanation for the plateau-like pattern of perfect
integer quantizations observed in the Integer Quantum-Hall Effect \cite%
{Nobel}, when the electron filling factor of the observed sample saturates
its bounded region degeneracy (corresponding to the dimension of the sample
as a vector space conditioned by the Pauli exclusive allocation of an
electron per unit cell). The integer jumps are thus explained by a simple
electron density model, requiring no understanding of the higher order
correlations. Thus, they can be described by the behavior of the single
particle electron density ($1$-point intensity), which has been shown to
have universal circular-law type behavior in all Landau levels, as a special
case of the limit law of finite Weyl-Heisenberg ensembles \cite{abgrro17}, a
geometrical flexible indicator function, initially studied with the aim of
approximating localization regions (or, in the language of anti-Wick
operators, of quantization symbols) by sums of spectrograms with windows
orthogonal in a particular strong sense \cite{AGR}.

Back to the recurring theme of the paper, the critical points of (\ref{GEF})
coincide with the zeros of the random analytic function%
\begin{equation}
F_{1}(z)=\nabla _{z}^{\prime }F(z)\text{, }  \label{F1}
\end{equation}%
with correlation kernel%
\begin{equation}
K_{1}(z,w)=\mathbb{E}(F_{1}(z)\overline{F_{1}(w)})=\nabla _{z}^{\prime }%
\overline{\nabla _{w}^{\prime }}e^{z\overline{w}}=L_{1}(\left\vert
z-w\right\vert ^{2})e^{z\overline{w}}\text{.}  \label{corrK1}
\end{equation}%
The Landau operator acting on the Hilbert space $L^{2}\left( \mathbb{C}%
,e^{-\left\vert z\right\vert ^{2}}\right) $ can be defined as minus the
composition of the symmetric and non-symmetric part of the Chern connection: 
\begin{equation}
L_{z}:=-\partial _{z}\partial _{\overline{z}}+\overline{z}\partial _{%
\overline{z}}=-\nabla _{z}^{\prime }\nabla _{z}^{\prime \prime }\text{.}
\label{2.1.3}
\end{equation}%
The spectrum of $L_{z}$ is given by $\sigma (L_{z})=\{r:r=0,1,2,\ldots \}$.

We now observe that the symmetric part of the Chern connection, $\nabla
_{z}^{\prime }=\partial _{z}-\overline{z}$ is precisely the inter Landau
levels raising operator (maps eigenfunctions of (\ref{2.1.3}) with
eigenvalue $r$ to eigenfunctions with eigenvalue $r+1$), while the
anti-symmetric part $\nabla _{z}^{\prime \prime }=\partial _{\overline{z}}$
is the corresponding lowering operator. Denote by $\mathcal{F}_{2}^{r}(%
\mathbb{C}
)$\ the space of functions which can be written as 
\begin{equation*}
f_{r}(z)=\left( \nabla _{z}^{\prime }\right) ^{r}f(z)\text{,}
\end{equation*}%
for some $f\in \mathcal{F}_{2}^{r}(%
\mathbb{C}
)$. Then $\mathcal{F}_{2}^{r}(%
\mathbb{C}
)$\ is the eigenspace of $L_{z}$ associated with the eigenvalue $r$ and the
corresponding reproducing kernel is%
\begin{equation*}
K_{r}(z,w)=\left( \nabla _{z}^{\prime }\right) ^{r}\overline{\left( \nabla
_{w}^{\prime }\right) ^{r}}e^{z\overline{w}}=r!L_{r}(\left\vert
z-w\right\vert ^{2})e^{z\overline{w}}\text{.}
\end{equation*}%
Setting $r=0$ gives the correlation kernel of the GEF (\ref{GEF}), while $%
r=1 $ leads to the correlation kernel (\ref{corrK1}) of the Chern connection
critical points for the random Gaussian function \ref{F1}. As a result, the
distribution of critical points of a GEF coincides with the distribution of
zeros of a Gaussian random function in the second Landau level ($r=1$).

\subsection{Spectrograms of white noise with the first Hermite window}

The random function $F_{1}(z)$ can also be understood as the STFT of white
noise, since%
\begin{equation*}
Spec_{g}\mathcal{W}(z_{c})=\left\vert V_{h_{1}}\mathcal{W}(\frac{\bar{z}}{%
\pi })\right\vert ^{2}=\left\vert e^{-\left\vert z\right\vert
^{2}/2}\sum_{k=0}^{\infty }a_{k}\frac{(k-\left\vert z\right\vert ^{2})z^{k-1}%
}{\sqrt{k!}}\right\vert ^{2}=\left\vert e^{-\left\vert z\right\vert
^{2}/2}F_{1}(z)\right\vert ^{2}\text{.}
\end{equation*}%
Thus, Theorem 1 implies that $Spec_{h_{1}}\mathcal{W}(z)$ has an average
number of $5/3$ zeros. Heuristically, one can think that the zero-crossings
of the Hermite window $h_{1}$ will be reflected in the spectrogram, leading
to an increase of the number of zeros. Following this reasoning, it is
likely to expect a monotone increasing in the average number of zeros of the
spectrogram with the order (which is fine tuned with the number of real
zero-crossings) of the Hermite window. These heuristics have been confirmed
in \cite[Corollary 1.10]{GWHF}, in work done simultaneously and
independently from this, where the result $r+1/2+\frac{1}{4r+2}$ $\ $is
obtained for the the average number of zeros of $V_{h_{r}}\mathcal{W}(\frac{%
\bar{z}}{\pi })$.

\begin{remark}
By Proposition 1, all local maxima of $\left\vert V_{h_{0}}\mathcal{W}(\bar{z%
})\right\vert $ are zeros of $\left\vert V_{h_{1}}\mathcal{W}(\bar{z}%
)\right\vert $. Thus, the local maxima of $\left\vert V_{h_{0}}\mathcal{W}(%
\bar{z})\right\vert $ and of $\left\vert V_{h_{1}}\mathcal{W}(\bar{z}%
)\right\vert $ are intertwined, confirming the observations in \cite[Figure
10.4]{Flbook} for $r=1$. However, it still remains an open question how to
proof the suggested intertwining property of local maxima for general $r$.
\end{remark}

\subsection{Gaussian Weyl-Heisenberg functions}

In this section we will consider the usual assumption $a_{k}\sim N_{\mathbb{C%
}}(0,1)$ i.i.d.\ The previous section suggests considering the STFT of white
noise with a general normalized window $g\in L^{2}({\mathbb{R}})$:%
\begin{equation*}
V_{g}\mathcal{W}(\frac{\bar{z}}{\pi })=\sum_{k=0}^{\infty }a_{k}V_{g}h_{k}(%
\frac{\bar{z}}{\sqrt{\pi }})\text{.}
\end{equation*}%
This leads us to a family of Gaussian functions with correlation kernel 
\begin{equation*}
{K}_{g}(z,z^{\prime })=V_{g}g(\frac{\bar{z}}{\pi }-\frac{\overline{z^{\prime
}}}{\pi })=\int_{\mathbb{R}}e^{-2i(y^{\prime }-y)t}g(t-x^{\prime }/\pi )%
\overline{g(t-x/\pi )}dt\text{.}
\end{equation*}%
The first two intensities of Gaussian Weyl-Heisenberg functions\ have been
computed in \cite{GWHF}. Setting $g=h_{0}$ we obtain, up to a phase factor,
the correlation kernel of the GEF (corresponding to the Fock kernel, or the
lowest Landau level):%
\begin{equation*}
{K}_{h_{0}}(z,z^{\prime })=e^{i(x^{\prime }y^{\prime }-xy)}e^{-\frac{1}{2}%
(\left\vert z\right\vert ^{2}-\left\vert z^{\prime }\right\vert ^{2})}e^{z%
\overline{z^{\prime }}}\text{.}
\end{equation*}%
The choice $g=h_{r}$ leads to a similar relation with the higher Landau
level kernels: 
\begin{equation*}
{K}_{h_{r}}(z,z^{\prime })=e^{i(x^{\prime }y^{\prime }-xy)}e^{-\frac{1}{2}%
(\left\vert z\right\vert ^{2}-\left\vert z^{\prime }\right\vert ^{2})}e^{z%
\overline{z^{\prime }}}L_{r}(\left\vert z-z^{\prime }\right\vert ^{2})\text{.%
}
\end{equation*}%
\ Determinantal point processes with the kernel ${K}_{g}(z,z^{\prime })$
have been studied in \cite{APRT,abgrro17}, under the name of Weyl-Heisenberg
ensembles.

In general, the non-analiticity of the basis functions creates an
obstruction in several of the arguments used in the fundamental results
about Gaussian Entire Functions, as it should be clear for a reader familiar
with the engrossing presentation of the topic in \cite{GAFbook}. For
instance, the arguments used in Proposition 1 break down if we try to look
at the Chern connection critical points of $V_{g}\mathcal{W}$ when $g$ is
not a Gaussian (the only window leading to entire functions, see \cite%
{AscensiBruna}). Thus, while it is clear how to obtain the zero-intensities,
the analysis leading to the statistics of local extrema may pose some
obstructions when trying to separate local minima from local maxima for more
general windows.

\subsection{Short-time Fourier transforms with white noise windows}

In some signal analysis applications it may be useful to consider a random
window, as done in \cite[Section 8]{PG,Goetz}, where a `white noise window'
has been used in the context of compressed sensing in the finite Gabor
setting. In the continuous infinite setting considered in this paper,
writing $f\mathcal{(}t\mathcal{)}=\sum_{k=0}^{\infty }a_{k}h_{k}(t)$, and
observing that 
\begin{equation*}
V_{g}f(x,\xi )=\int_{{\mathbb{R}}}f(t)\overline{g(t-x)}e^{-2\pi i\xi
t}dt=\int_{{\mathbb{R}}}f(t+x)\overline{g(t)}e^{-2\pi i\xi (t+x)}dt=e^{-2\pi
i\xi x}V_{f}g(x,\xi )
\end{equation*}%
we can write the STFT modulus of $f$ with a random white noise window $%
\mathcal{W}$ as a Gaussian Weyl-Heisenberg function with a window $f$: 
\begin{equation*}
\left\vert V_{\mathcal{W}}f(\frac{\bar{z}}{\pi })\right\vert =\left\vert
V_{f}\mathcal{W}(\frac{\bar{z}}{\pi })\right\vert \text{,}
\end{equation*}%
leading to the set-up of the previous section.

A related problem of interest in signal analysis is the study of the STFT
transform of non-white random noise. Some heuristics regarding the
invariance of the zeros of Gaussian spectrogram noise under `coloring' can
be found in \cite[page 156]{Flbook}.

\subsection{Gaussian bi-entire functions}

A further observation will lead us to the consideration of a related
Gaussian non-analytic function. A complex valued function is said to be
polyanalytic of order $n$ if it satisfies the higher order Cauchy-Riemann
equations,%
\begin{equation}
\partial _{\overline{z}}^{n}f=0\text{.}  \label{polyanalytic}
\end{equation}%
The Fock space of polyanalytic functions, $\mathbf{F}^{n}(%
\mathbb{C}
)$ is the space of polyanalytic functions of order $n$, with finite Fock
norm (\ref{Focknorm}). Vasilevski \cite{VasiFock} obtained the following
decomposition of $\mathbf{F}^{n}(%
\mathbb{C}
)$, in terms of the Landau levels eigenspaces $\mathcal{F}_{2}^{r}(%
\mathbb{C}
)$:%
\begin{equation}
\mathbf{F}^{n}(%
\mathbb{C}
)=\mathcal{F}_{2}^{0}(%
\mathbb{C}
)\oplus ...\oplus \mathcal{F}_{2}^{n-1}(%
\mathbb{C}
)\text{.}  \label{orthogonal}
\end{equation}

Despite some recent research activity related to determinantal point
processes in higher Landau levels and polyanalytic Fock spaces \cite%
{HaiAron,APRT,abgrro17,HendHaimi,SHIRAI,MakotoShirai,Deviations}, we were
unable to find any literature about the corresponding Gaussian random
(non-analytic) functions with exception of  \cite[Theorem 1.8]{GWHF} (though
they occurred implicitly in some form in \cite{NS,DZS}, as we have made
clear in the previous section).

We will finish this note by indicating how to compute the expected number of
zeros of a Gaussian bi-entire function $\mathbf{F(}z\mathbf{)}$ (\ref%
{bientire}). By bi-entire we mean polyanalytic of order $2$ in $\mathbb{C}$%
,\ according to (\ref{polyanalytic}). The decomposition (\ref{orthogonal})
suggests considering the Gaussian random function%
\begin{equation}
\mathbf{F(}z\mathbf{)=}F(z)+F_{1}(z)\mathbf{=}\sum_{k=0}^{\infty }\left(
a_{k}^{(0)}\frac{z^{k}}{\sqrt{k!}}+a_{k}^{(1)}\frac{(k-\left\vert
z\right\vert ^{2})z^{k-1}}{\sqrt{k!}}\right)  \label{bientire}
\end{equation}%
which, according to the models in previous sections, represents the sum of
two STFT's of independent realizations of white noise, 
\begin{equation*}
\mathcal{W}^{(0)}=\sum_{k=0}^{\infty }a_{k}^{(0)}h_{k}(t);\text{ \ \ \ \ \ \
\ \ \ \ \ \ }\mathcal{W}^{(1)}=\sum_{k=0}^{\infty }a_{k}^{(1)}h_{k}(t)\text{,%
}
\end{equation*}%
transformed with orthogonal windows $h_{0}$ and $h_{1}$:%
\begin{equation*}
V_{h_{0}}\mathcal{W}^{(0)}(\frac{\bar{z}}{\sqrt{\pi }})+V_{h_{1}}\mathcal{W}%
^{(1)}(\frac{\bar{z}}{\sqrt{\pi }})=e^{\pi ix\xi }e^{-\left\vert
z\right\vert ^{2}/2}\mathbf{F(}z\mathbf{)}\text{.}
\end{equation*}%
The associated correlation kernel is the sum of the correlation kernel of
the zeros of $F(z)$, with the correlation kernel of the zeros of $F_{1}(z)$: 
\begin{equation*}
K(z,w)=L_{1}^{(1)}(\left\vert z-w\right\vert ^{2})e^{z\overline{w}}=e^{z%
\overline{w}}+L_{1}(\left\vert z-w\right\vert ^{2})e^{z\overline{w}}\text{.}
\end{equation*}%
This is the reproducing kernel of the Fock space of bi-analytic functions $%
\mathbf{F}^{2}(%
\mathbb{C}
)$, given by the sum of the kernels of the first two Landau levels
eigenspaces. The average number of zeros of $F\mathbf{(}z\mathbf{)}$ can be
computed with a procedure similar to the one used in the proof of Theorem 1,
by considering the 3-dimensional Gaussian process 
\begin{equation*}
(F(z),\nabla _{z}^{\prime }F(z),\nabla _{z}^{\prime \prime }F(z))
\end{equation*}%
with correlation kernel $K(z,w)$. One can also use the correspondences
between polyanalytic functions and time-frequency analysis \cite%
{Abr2010,abgrro17} and deduce the result from the theory of Gaussian
Weyl-Heisenberg functions (see \cite[Theorem 1.8]{GWHF}). A related problem
is the study of bi-analytic functions with more general weights. For this
purpose, the asymptotic expansion of the bi-analytic Bergman kernel obtained
in \cite{HaiHen2} may come in handy.

\section{Conclusion}

Based on the correspondence between Gaussian entire functions and
spectrograms of white noise explored in \cite{BFC,BH}, we have provided a
rigorous proof that spectrograms of white noise with Gaussian windows have
an average of $5/3$ critical points, among those $1/3$ are local maxima, and 
$4/3$ saddle points, providing, simultaneously, information about the number
of neighbors of basins of zeros and local extrema, in the model of Nazarov,
Sodin, and Volberg \cite{BO2}. This confirms what is predicted by Flandrin's
double honeycomb mean mode for the average of local maxima and zeros of
white noise spectrograms, where the area of local maxima centered hexagons
is three times larger than the area of zero centered hexagons. We suggest
that such mean model may provide a heuristic reason for the factor 1/3 in
the topological terms of similar problems in complex Riemann surfaces, a
question asked by Douglas, Shiffman and Zelditch in \cite{DZS}. Using
methods of Feng and Zelditch \cite{FZ,FZ1}, we have also provided the
distribution of the ordinates of such points. The information may be useful
in signal recovering threesholding-based methods and in algorithms based on
spectrogram maxima in the style of the publicy released Shazam
Entertainment's recognition algorithm \cite{Wang}. We also made a connection
to Euclidean Landau levels, a model used in condensed matter physics related
to the Quantum-Hall Effect: by noting that the translation invariant part of
the Chern connection coincided with the inter-Landau levels raising
operator, we pointed out that the critical points of the GEF are the zeros
of a Gaussian random function with correlations given by the reproducing
kernel of the first higher Landau level eigenspace, an instance of a space
of bi-entire functions. This suggested the introduction of a bi-entire
Gaussian random function which represents the sum of two independent
realizations of white noise, STFT transformed with the first two Hermite
windows. Further directions with applied potential in physics and signal
analysis involve the setting of compact Riemann surfaces covered in \cite%
{DZS}. In particular, $SU(2)$ polynomials provide finite dimensional models
with treatable formulas \cite{DZS,FZ,FZ1}, which may find applications in
the context of \cite{Sphere}. The $2$-point intensities of critical points
seem to be untreatable without computer algebra support \cite{Baber,GWHF}.

Finally, a few comments which, before risking contempt by more
scientifically orthodox thinkers, should be considered within the umbrella
of speculative assertions. The analogies between the statistics of
supersymmetric vacua, spectrogram extrema and the structure of zeros and
critical points could suggest using the combination of acoustic-related
numerical and thought experiments, of the kind presented in the recent
monograph \cite{Flbook}, as a source of intuition to large scale macroscopic
and microscopic phenomena. The combination of ideas used in this article
intertwine basic principles of time-frequency representations, Gaussian
analytic functions, and condensed matter physics. One may accept the
possibility of a cross-fertilization between such topics, which,
traditionally, tend to be treated in different subdisciplines of
mathematics, physics and acoustics. Moving a step further in these
speculations, one could point to a more abstract picture, since all these
topics may be brought together by the formalism of noncommutative geometry 
\cite{Connes}. For an overview of some of the links between time-frequency
analysis and noncommutative geometry, see for instance the survey of Luef
and Manin \cite{LuefManin}.

\begin{acknowledgement}
The author wishes to thank Patrick Flandrin for motivating discussions and
encouragement in pursuing some of the fascinating questions suggested by his
recent monograph \cite{Flbook}. This paper evolved from a 2020 RG manuscript
which contained several incorrections in its original form, specially in the
proof of Theorem 2. I would like to thank the patience of those who struggled
to read the first versions, and I'm particularly indebted to Guenther
Koliander, Michael Speckbacher and Tomoyuki Shirai for discussions and
valuable feedback during the preparation of this manuscript, to Mikhail
Sodin for clarifying the current (open) status of the questions in \cite[12.3]{BO2}%
, and to Arnaud Poinas for correcting several mistakes. The reviewing
process has considerably improved the paper and I'm grateful for the
editorial efforts and the valuable comments and corrections provided by the
three reviewers.
\end{acknowledgement}

\end{document}